\documentclass{amsart}
\newtheorem{theorem}{Theorem}[section]
\newtheorem{lemma}[theorem]{Lemma}
\theoremstyle{definition}

\newtheorem{pro}[theorem]{Proposition}
\newtheorem{remark}[theorem]{Remark}
\newtheorem{cor}[theorem]{Corollary}
\numberwithin{equation}{section}

\begin{document}

\title{sums of two biquadrates and elliptic curves of rank $\geq 4$}

\author{F.A. Izadi}
\address{Mathematics Department Azarbaijan university of  Tarbiat Moallem ,
Tabriz, Iran}
\email{f.izadi@utoronto.ca}

\author{F. Khoshnam}
\address{Mathematics Department Azarbaijan university of  Tarbiat Moallem ,
 Tabriz, Iran}
\email{khoshnam@azaruniv.edu}
\author{K. Nabardi}
\address{Mathematics Department Azarbaijan university of  Tarbiat Moallem ,
 Tabriz, Iran}
\email{nabardi@azaruniv.edu}

\date{Februry 25, 2012.}



\begin{abstract}
If an integer $n$ is written as a sum of two biquadrates in two different ways, then the elliptic curve $y^2=x^3-nx$
has rank $\geq 3$. If moreover $n$ is odd and the parity conjecture is true, then it has even rank
  $\geq 4$.
Finally, some examples of ranks equal to  $4$, $5$, $6$, $7$, $8$ and $10$, are also obtained.
\end{abstract}
\maketitle



\section{Introduction}

Let $E$ be an elliptic curve over $\Bbb{Q}$ defined by the Weierstrass equation of the form
\begin{equation}
E: y^2=x^3+ax+b \quad\    a,b \in \Bbb{Q}.
\end{equation}
In order the curve $(1.1)$ to be an elliptic curve, it must be smooth. This in turn is equivalent to requiring that the cubic on the right of Eq. $(1.1)$ have no multiple roots. This holds if and only if the $discriminant$ of $x^3+ax+b$, i.e., $\Delta=-16(4a^3+27b^2)$ is non-zero.\\
By the Mordell-Weil theorem, the set of rational points on $E$ i.e., $E(\Bbb{Q})$ is a finitely generated abelian group, i.e.,
$$E(\Bbb{Q})\simeq E(\Bbb{Q})_{{\rm{tors}}}\oplus \Bbb{Z}^r,$$ where $E(\Bbb{Q})_{{\rm{tors}}}$ is a finite group called the torsion group and $r$ is a non-negative integer called the Mordell-Weil rank of $E(\Bbb{Q})$.
In this paper, we consider the family of elliptic curves defined by
$$E_n : y^2=x^3-nx,$$ for positive integers $n$ written as  sums of two biquadrates in two
different ways, i.e., $$n=p^4+q^4=r^4+s^4,$$ where the greatest common divisor of all the numbers $p, q, r, s$ is one.
Such a solution is referred to as a primitive solution. In what follows we deal only with numbers $n$ having primitive solution. This Diophantine equation was first proposed by Euler \cite{Euler} in 1772 and has since aroused the interest of numerous mathematicians. Among quartic Diophantine equations it has a distinct feature for its simple structure, the almost perfect symmetry between the variables and the close relationship with the theory of elliptic functions. The latter is demonstrated by the fact that this equation is satisfied by the four elliptic theta functions of Jacobi, $\vartheta_1, \vartheta_2, \vartheta_3, \vartheta_4$, in that order \cite{Wh}. Here in this note, we show that it also has an obvious relationship with the theory of elliptic curves. To this end, we need some
parametric solutions of the equation for which we use the one that was constructed by Euler as:
 \begin{equation}
\left\{
  \begin{array}{ll}
   p=a^7+a^5b^2-2a^3b^4+3a^2b^5+ab^6, \\
   q=a^6b-3a^5b^2-2a^4b^3+a^2b^5+b^7,\\
   r=a^7+a^5b^2-2a^3b^4-3a^2b^5+ab^6,\\
   s=a^6b+3a^5b^2-2a^4b^3+a^2b^5+b^7.\\
  \end{array}
\right.
\end{equation}
(See Hardy and Wright \cite{Ha} page 201, problem no.(13.7.11)). It is easy to see that the two different integers $n_1$ and $n_2$ having primitive solutions are independent modulo $\Bbb{Q}^{\ast 4}.$ For let $n_1$ and $n_2$ be two such numbers in which $(p_1, q_1, r_1, s_1)$ is the solution for $n_1$ and  $n_2=k^4n_1$ for non-zero rational number $k$. It follows that $(kp_1, kq_1, kr_1, ks_1)$ is a solution for $n_2$ which is not primitive.
We see that this condition is sufficient for the curves $E_{n_1}$ and $E_{n_2}$
to be non-isomorphic over $\Bbb{Q}$ (the dependence modulo $\Bbb{Q}^{\ast k}$ for $k=0, 1, 2, 3$ expresses one curve as the quartic twists of the other). However, it is not plain that there are infinitely many integers having primitive solution. To remedy this difficulty, Choudhry \cite{ch} presented a method of deriving new primitive solutions starting from a given primitive solution. This makes it possible to construct infinitely many non-isomorphic elliptic curves using the primitive solutions of the biquadrate equation.
Our main results are the following:

\begin{theorem}
If an integer $n$ is written as a sum of two biquadrates in two different ways, then the elliptic curve $y^2=x^3-nx$
has rank $\geq 3$. If moreover $n$ is odd and the parity conjecture is true, then it has even rank
$\geq 4$.
\end{theorem}
\begin{remark}
Our numerical results suggest that the odd ranks for even numbers should be at least 5.
\end{remark}

\section{Previous works}
For questions regarding the rank, we assume without loss of generality that
$n \not\equiv 0\pmod{4}$. This follows from the fact that $y^2=x^3-nx$ is a $2$-isogenous to $y^2=x^3+4nx$. These curves form a natural family in the sense that they all have $j$-invariant $j(E)=1728$ regardless of the different values or various properties that the integers $n$ may have. There have been a lot of investigations concerning the distribution of ranks of elliptic curves in natural families, and it is believed that the vast majority of elliptic curves $E$ over $\Bbb{Q}$ have rank
 $\leq 1$. Consequently, the identification of elliptic curves of rank $\geq 2$ is of great interest.
\par Special cases of the family of the curves $E_n$ and their ranks have been studied by many authors including Bremner and Cassels \cite{Br}, Kudo and Motose \cite{Ko}, Maenishi \cite{Ma}, Ono and Ono \cite{On}, Spearman \cite{Spe,Spe2}, and Hollier, Spearman and Yang \cite{Ho}. The general cases were studied by Aguirre, Castaneda, and Peral \cite{Ag}.
\par The main purpose of Aguirre et al., \cite{Ag} was to find the elliptic curves of high rank in this family without restricting $n$ to have any prescribed property. They developed an algorithm for general $n$, and used it to find 4 curves of rank 13 and 22 of rank 12.
\par Breamner and Cassels \cite{Br} dealt with the case $n=-p$, where $p\equiv5\pmod{8}$ and less than $1000.$ The rank is always 1 in accordance with the conjecture of Selmer and Mordell. For each prime in this range, the authors found the generator for the free part. In some cases the generators are rather large, the most startling being that for $p=877$, the $x$ has the value
$$x=\left(\frac{612776083187947368101}{7884153586063900210}\right)^2.$$
\par Kudo and Motose \cite{Ko} studied the curve for $n=p$, a Fermat or Mersenne prime and found ranks of $0$, $1$, and $2$. More precisely,

\begin{enumerate}
\item For a Fermat prime $p=2^{2^n}+1,$
\begin{equation*}
E(Q)\cong\left\{\begin{array}{ll}
\Bbb{Z}/2\Bbb{Z}& {\rm for} \ p=3\\
\Bbb{Z}/2\Bbb{Z}\oplus\Bbb{Z}& {\rm for} \ p=5\\
\Bbb{Z}/2\Bbb{Z}\oplus\Bbb{Z}\oplus\Bbb{Z}& {\rm for} \ p>5.
\end{array}\right.
\end{equation*}
\item In case $p=2^q-1$ is a Mersenne prime where $q$ is a prime,
\begin{equation*}
E(Q)\cong\left\{\begin{array}{ll}
\Bbb{Z}/2\Bbb{Z}& {\rm for} \  p=3\\
\Bbb{Z}/2\Bbb{Z}\oplus \Bbb{Z}& {\rm for} \  p>3.
\end{array}\right.
\end{equation*}
\end{enumerate}

Maenishi \cite{Ma} investigated the case $n=pq$, where $p,q$ are distinct odd primes and found a condition that the rank of $E_{pq}$ equals 4. This can be done by taking natural numbers $A, B, C, D$ and two pairs $p$ and $q$ satisfying the equations:
$$pq=A^2+B^2=2C^2-D^4=S^4-4t^4\qquad (p=s^2-2t^2, q=s^2+2t^2).$$
Then using these equations one can construct 4 independent points on the corresponding elliptic curve.\\
In \cite{On} the authors examined the elliptic curves for $n=b^2+b$, where $b\neq 0,-1$ is an integer, and show that, subject to the parity conjecture, one can construct infinitely many curves $E_{b^2+b}$ with even rank $\geq 2$. To be more precise they obtained the followings:
\par
\vspace{.2cm}
{\it Let $b\neq 0,-1$  be an integer for which  $n=b^2+b$, is forth power free, and define $T$ by
$$T:={\rm{card}}\{p\ |\  primes \ 3\leq p\equiv3 \pmod {4},\  p^2\parallel b^2+b\}.$$
\begin{itemize}
\item[1.] If $b\equiv1,2\pmod{4}$ and $T$ is odd, then $E(b)$ has even rank $\geq$2.
\item[2.]If $b\equiv 7,8,11,12,20,23,24,28,35,39,40,43,51,52,55,56\pmod{64}$ and $T$ is even, then $E(b)$ has even rank $\geq$2.
\item[3.] If $b\equiv 3,14,19,27,36,44,59,60\pmod{64}$ and $T$ is odd, then $E(b)$ has even rank $\geq$2.
\item[4.] In all other cases, $E(b)$ has odd rank.
\end{itemize}
\par \vspace{.2cm}}
In two separate papers, Spearman \cite{Spe}, \cite{Spe2} gave the following two results:\\
(1) If $n=p$ for an odd prime $p$ written as $p=u^4+v^4$ for some integers $u$ and $v$, then
$$E(\Bbb{Q})={\Bbb Z}/2{\Bbb Z}\oplus{\Bbb Z}\oplus{\Bbb Z}.$$
(2) If $n=2p,$ where $2p=(u^2+2v^2)^4+(u^2-2v^2)^4$ for some integers $u$ and $v$, then
$$E(\Bbb{Q})={\Bbb Z}/2{\Bbb Z}\oplus{\Bbb Z}\oplus{\Bbb Z}\oplus\overline{}{\Bbb Z}.$$
In recent paper Spearman along with Hollier and Yang \cite{Ho} assuming the parity conjecture  constructed elliptic curves of the form $E_{-pq}$ with maximal rank 4, here   $p\equiv1\pmod{8}$ and $q$ be an odd prime different from $p$ satisfying
\begin{equation*}
q=p^2+24p+400.
\end{equation*}
Finally, Yoshida \cite{Yosh} investigated the case $n=-pq$ for distinct odd primes $p, q$ and showed that for general such $p,q$ the rank is at most 5 using the fact that
$${\rm{rank}}(E_n(\Bbb{Q}))\leq 2\#\{l\ {\rm prime;\ divides}\ 2n\}-1.$$
If $p$ is an odd prime, the rank of $E_p(\Bbb{Q})$ is much more restricted, i.e.,
\begin{equation*}
{\rm{rank}}(E_p(\Bbb{Q}))\leq \left\{\begin{array}{lll}
0&{\rm{if}}&p\equiv7,11\pmod{16}\\
1&{\rm{if}}&p\equiv 3,5,13,15 \pmod{16}\\
2&{\rm{if}}& p\equiv 1\pmod{8}.
\end{array}\right.
\end{equation*}
\par If the  Legendre symbol $(q/p)=-1$ and $q-p\equiv\pm6\pmod{16}$, then
\begin{equation*}
E_{-pq}(\Bbb{Q})=\{\mathcal{O},(0,0)\}\cong\Bbb{Z}/2\Bbb{Z}.
\end{equation*}
If $p, q$ are twin prime numbers, then $E_{-pq}(\Bbb{Q})$ has a non-torsion point $(1,(p+q)/2)$.
If $p, q$ be twin primes with $(q/p)=-1,$ then
\begin{equation*}
E_{pq}(\Bbb{Q})\cong\Bbb{Z}\oplus{\Bbb Z}/2\Bbb{Z}.
\end{equation*}
Having introduced the previous works, one can easily see that all the elliptic curves including those in our family share three main properties in common. They have the same $j$-invariant $j(E)=1728$, have positive rank (except for the case $p=3$  in the Kudo-Motose \cite{Ko}
paper with rank zero), and have the torsion group $T=\Bbb{Z}/2\Bbb{Z}$, as we will see in the next section. In spite of these similarities our family has almost higher ranks among all the
other families and can be taken as an extension of the previous results. Before we proceed to the proofs, we wish to make the following remarks.
\begin{remark}
Our result for odd $n$ is conditional on the parity conjecture. In \cite{ap} the authors using the previous version of our work proved the following two results unconditionally.
\par \vspace{.2cm}
Theorem 1. {\it The family $y^2=x^3-nx,$ with $n=p^4+q^4$ has rank at least 2 over $\Bbb{Q}(p,q).$}
\par\vspace{.2cm}
Theorem 2. {\it The family $y^2=x^3-nx,$ in which $n$ given by the Euler parametrization has rank at least $4$ over $\Bbb{Q}(a)$, where $a$ is the parameter and $b=1.$}
\vspace{.2cm}
\par One may prove both results by a very straightforward way. For the first theorem, we note that, by the same
reasons as in \cite{ap} not only the point $Q(p,q)=(-p^2,pq^2)$, but also the point $R(p,q)=(-q^2,qp^2)$ is on the curve. Then
the specialization by $(p,q)=(2,1)$ gives rise to the points $Q=(-4, 2)$ and $R=(-1, 4)$. Therefore
by using the Sage software, we see that the associated height matrix has non-zero determinant $ 1.8567  $ showing that the points are independent.
For the second theorem, we see that the points $Q_1(a)=(-p^2,pq^2),$ $Q_2(a)=(-q^2,qp^2),$ $Q_3(a)=(-r^2,rs^2)$ and
$Q_4(a)=(-s^2,sr^2)$ are on the curve and the specialized points for $a=2$ gave rise to
$$Q_1=(-24964, 549998), \ Q_2=(-3481, -1472876),$$
\vspace{-.5cm}
 $$Q_3=(-17956, 2370326), \ Q_4=(-17689, 2388148).$$
By using the Sage software we find that the elliptic height matrix associated to $\{Q_1,Q_2,Q_3,Q_4\}$ has non-zero determinant
 $5635.73654 $ showing that again the 4 points are independent.
\end{remark}
\begin{remark}
We see that the map $(u,v)\rightarrow(-u^2, uv^2 )$ from the quadric curve: $u^4+v^4=n$ to the elliptic curve:
 $y^2=x^3-nx$ takes the integral points of the first to the integrals of the second. Now to find the integral
points of the quadric, it is enough to find the integrals of the elliptic curve. This might suggests that to find
$n$ with more representations as sums of two biquadrates, the corresponding elliptic curve should have many independent
integral points.
\end{remark}
\section{Method of Computation} To compute the rank of this family of elliptic curves, a couple of facts are necessary from the literature. We begin by describing the torsion group of the family. To this end, let $D\in \Bbb{Z}$ be a fourth-power-free integer, and let $E_D$ be the elliptic curve
\begin{equation*}
E_D: y^2=x^3+Dx.
\end{equation*}
Then we have
\begin{equation*}
E_D(\Bbb{Q})_{{\rm{tors}}}\cong\left\{
\begin{array}{ll}
\Bbb{Z}/4\Bbb{Z}&{\rm{if}}\ D=4,\\
&\\
\Bbb{Z}/2\Bbb{Z}\times\Bbb{Z}/2\Bbb{Z}&{\rm{if}}\  -D\ {{\rm{is \ a \ perfect\  square}}},\\
&\\
\Bbb{Z}/2\Bbb{Z}&\mbox{otherwise}.\\
\end{array}\right.
\end{equation*}
See( \cite{Silv}  Proposition 6.1, Ch.X, page 311). Since $n=p^4+q^4$ is not $-4$ and can not be a square, (see for example \cite{coh}, proposition 6.5.3, page 391), we conclude that the family has the torsion
group $T=\Bbb{Z}/2\Bbb{Z}.$\\

\par The second fact that we need is  the parity conjecture which takes the following explicit form (see Ono and Ono \cite{On}).\\
Let $r$ be the rank of elliptic curve $E_n$, then $$(-1)^r=\omega(E_n)$$ where $$\omega(E_n)={\rm{sgn}}(-n)\cdot\epsilon(n)\cdot\prod_{p^2||n}\left(\frac{-1}{p}\right)$$
with $p\geq 3$ a prime and
\begin{equation}
\epsilon(n)=\left\{
  \begin{array}{ll}
   -1, & n\equiv 1,3,11,13  \pmod{16}, \\
   &\\
    1, & n \equiv 2,5,6,7,9,10,14,15 \pmod{16}.\\
  \end{array}
\right.
\end{equation}
\vspace{.2cm}
As we see from the parity conjecture formula, the key problem is to calculate the product $\prod_{p^2||n}\left(\frac{-1}{p}\right).$
For this reason it is necessary to describe the square factors of the numbers $n$ if there is any. Before discussing
the general case, we look at some examples:\\

$(p,q,r,s)=(3364, 4849, 4288, 4303)$
with $17^2|n$,

$(p,q,r,s)=(17344243, 6232390, 12055757, 16316590)$
with $97^2|n$,

$(p,q,r,s)=(9066373, 105945266, 5839429, 105946442)$
with $17^2|n$,

$(p,q,r,s)=(160954948, 40890495, 114698177, 149599920)$ with $41^2|n$.\\
\par These examples show that the prime divisor of the square factor of $n$ are of the form $p=8k+1$.
We will see that this is in fact a general result according to the following proposition.
\begin{pro}
Let $n=u^4+v^4=r^4+s^4$ be such that $\gcd(u, v, r, s)=1$.
If $p|n$ for an odd prime number $p$, then $p=8k+1$.
\end{pro}
\begin{proof}
Without loss of generality we can assume that $n$ is not divisible by 4.
We use the following result from Cox \cite{cox}.
Let  $p$ be an odd prime such that $\gcd(p,m)=1$ and $p| x^2+my^2$ with
$\gcd(x,y)=1$, then $(\frac{-m}{p})=1$. From one hand
for $n=u^4+v^4=(u^2-v^2)^2+2(uv)^2$, we get $(\frac{-2}{p})=1$ which implies that
$p=8k+1$ or $p=8k+3$. On the other hand
for $n=u^4+v^4=(u^2+v^2)^2-2(uv)^2$, we get $(\frac{2}{p})=1$ which implies that
$p=8l+1$ or $p=8l+7$. Putting these two results together we get $p=8k+1$.
\end{proof}
\begin{remark}
If $n=p^2m$ for an odd prime $p,$ then $p=8k+1$ from which we get $(\frac{-1}{p})=1$. This last
result shows that the square factor of $n$ does not affect the root number of the
corresponding elliptic curves on the parity conjecture formula.
\end{remark}
\begin{remark}
First of all, by the above remark, we have
$$\omega(E_n)={\rm{sgn}}(-n)\cdot\epsilon(n).$$
On the other hand, for $n=p^4+q^4$, we note that
$$p^4\equiv 0\  {\rm{or}}\  1 \pmod{16},$$
$$q^4\equiv 0\  {\rm{or}}\  1  \pmod{16}.$$
For odd  $n$ we note that $$n\equiv 1\pmod{16}.$$
Now the parity conjecture implies that $$\omega(E_n)={\rm{sgn}}(-n)\cdot\epsilon(n)=(-1)\cdot(-1)=1.$$
For even $n$ we have $n\equiv 2\pmod{16}$ and therefore $\omega(E_n)=-1$ in this case.
\end{remark}

 Finally, we need the Silverman-Tate computation formula \cite{Sil} (Ch.3 \S.5, p.83) to compute the rank of this family. Let $G$ denote the group of rational points on elliptic curve $E$ in the form $y^2=x^3+ax^2+bx$. Let $\Bbb{Q}^\ast$ be the multiplicative group of non-zero rational numbers and let $\Bbb{Q}^{\ast 2}$ denote the subgroup of squares of elements of $\Bbb{Q}^\ast$. Define the group homomorphism $\phi$ from $G$ to ${\Bbb{Q}^\ast}/\Bbb{Q}^{\ast 2}$ as follows:
\begin{equation*}
\phi(P)=\left\{\begin{array}{lll}
1\pmod{\Bbb{Q}^{\ast 2}}& {\rm{if}} & P=\mathcal{O},\\
b\pmod{\Bbb{Q}^{\ast 2}}&{\rm{if}}& P=(0,0),\\
x\pmod{\Bbb{Q}^{\ast 2}}& {\rm{if}}& P=(x,y)\  {\rm with}\  x\not=0.
\end{array}\right.
\end{equation*}
Similarly we take the dual curve $y^2=x^3-2ax^2+(a^2-4b)x$ and call its group of rational points $\overline{G}.$
Now the group homomorphism $\psi$ from $\overline{G}$ to $\Bbb{Q}^\ast/\Bbb{Q}^{\ast 2}$ defined as
\begin{equation*}
\psi(Q)=\left\{\begin{array}{llll}
1&\pmod{\Bbb{Q}^{\ast 2}}& {\rm{if}}& Q=\mathcal{O},\\
a^2-4b&\pmod{\Bbb{Q}^{\ast 2}}&{\rm{if}}&Q=(0,0),\\
x&\pmod{\Bbb{Q}^{\ast 2}}&{\rm{if}}&Q=(x,y)\ {\rm{with}}\ x\not=0.
\end{array}\right.
\end{equation*}
Then the rank $r$ of the elliptic curve $E$ satisfies
\begin{equation}
2^{r+2}=|\phi(G)||\psi(\overline{G}|.
\end{equation}

\section{Proof of Theorem 1.1}
The following fact is an important tool in the proof of our main result.
\begin{lemma}
Let
\begin{eqnarray*}
A\hspace{-.25cm}&=&\hspace{-.25cm}b^4+6b^2a^2+a^4,\\
B\hspace{-.25cm}&=&\hspace{-.25cm}b^8+2b^6a^2+11b^4a^4+2b^2a^6+a^8,\\
C\hspace{-.25cm}&=&\hspace{-.25cm}b^8-4b^6a^2+8b^4a^4-4b^2a^6+a^8,\\
D\hspace{-.25cm}&=&\hspace{-.25cm}b^8-b^4a^4+a^8.
\end{eqnarray*}
We have the following properties:
\begin{itemize}
\item[1.] $B\neq D$.\\
\item[2.] $D$ is non-square.\\
\item[3.] $A\neq C$.\\
\item[4.] $A$ is non-square.
\end{itemize}
\end{lemma}

{\bf{Proof of lemma 4.1}}
Let $B=D.$ Since $ab\neq 0$, we get $b^4+6a^2b^2+a^4=0,$ which has no nontrivial solution.\\
For part 2, we consider the diophantine equation $x^4-x^2y^2+y^4=z^2,$ which has only the trivial solutions
$x^2=1, y=0$ and $y^2=1, x=0$ (see \cite{Mo} page 20). \\
If  $C=A,$ then  $(a^2-b^2)^2+8a^2b^2=(a^2-b^2)^4+2a^4b^4$. Setting  $u=(a^2-b^2),v=a^2b^2$ , we get
 $2(v-2)^2=-u^4+u^2+8$. This is an elliptic curve  with Weierstrass equation $y^2=x^3-x^2-129x-127$, and integral points $(-1,0)$, and  $(17,48)$.
Similarly, for part 4, we get the diophantine equation $x^4+6x^2y^2+y^4=z^2$
which has only the solutions $x^2=1, y=0$ and $y^2=1, x=0$ (see \cite{Mo} page 18).\\
The following corollary is an immediate consequence of the above lemma.

\begin{cor} Let $b_1=BD$, $b_2=-AC$, $n=-b_1b_2$, where $A$, $B$, $C$, and $D$ as in the above lemma, then
the elements of the sets $\{1,-n,-1,n,b_1,-b_1, b_2, -b_2\}$ and $\{1,2,n,2n\}$ are independent modulo
 ${\Bbb Q}^{\star 2}$.
\end{cor}
{\bf Proof of Corollary 4.2} Without loss of generality we check only independence of the positive  numbers in both sets. By construction we know that the numbers  $n$, $b_1$ and $-b_2$ are all non-squares. Let $b_1=c_1c_2^2$, and  $-b_2=d_1d_2^2$, where $c_1\geq2$,  $c_2\geq1$ , $d_1\geq2$, $d_2\geq1$ and $c_1\neq d_1$.
Since $n=-b_1b_2=c_1d_1(c_2d_2)^2=rs^2$ where $r>2$,  $r\notin {\Bbb Q}^{\ast 2}$, and $ s\geq1$, then we have

\begin{center}
$\begin{array}{l}
\displaystyle\frac{n}{1}=rs^2\equiv r\not\equiv1\pmod{{\Bbb Q}^{\ast 2}},\\
\\
\displaystyle\frac{n}{b_1}=-b_2=d_1d_2^2\equiv d_1\not\equiv1\pmod{{\Bbb Q}^{\ast 2}},\\
\\
\displaystyle\frac{n}{-b_2}=b_1=c_1c_2^2\equiv c_1\not\equiv1\pmod{{\Bbb Q}^{\ast 2}},\\
\\
\displaystyle\frac{b_1}{-b_2}=\left(\frac{c_1}{d_1}\right)\left(\frac{c_2}{d_2}\right)^2\equiv
\displaystyle\frac{c_1}{d_1}\not\equiv1\pmod{{\Bbb Q}^{\ast 2}},\\
\\
\displaystyle\frac{n}{2}=\frac{n}{2}=\frac{r}{2}s^2\equiv\frac{r}{2}\not\equiv1\pmod{{\Bbb Q}^{\ast 2}},\\
\\
\displaystyle\frac{2n}{1}=2rs^2\equiv2r\not\equiv1\pmod{{\Bbb Q}^{\ast 2}}\\
\\
\displaystyle\frac{2n}{n}=2\not\equiv1\pmod{{\Bbb Q}^{\ast 2}}.
\end{array}$
\end{center}

{\bf {Proof of theorem 1.1}}
First of all, we show that
\begin{equation*}
\phi(G)\supseteq\{1,-n,-1,n\}.
\end{equation*}
The first two numbers $1$ and $-n$ are obvious from the definition of the map $\phi$. For the numbers $-1$ and $n$ we note that if $n=p^4+q^4$, then the homogenous equation
\begin{equation*}
N^2=-M^4+ne^4
\end{equation*}
has solution $e=1$, $M=p$, $N=q^2$. Similarly for $N^2=nM^4-e^4$ we have $M=1$, $e=p$, $N=q^2$.\\
Next we know that the Euler parametrization for $n$ is a consequence of the fact that $n=p^4+q^4=r^4+s^4$ for different numbers $p, q, r, s$. This implies that
\begin{align*}
n=&(b^4+6b^2a^2+a^4)(b^8+2b^6a^2+11b^4a^4+2b^2a^6+a^8)\\&(b^8-4b^6a^2+8b^4a^4-4b^2a^6+a^8)(b^8-b^4a^4+a^8).
\end{align*}
Let $b_1=BD$, $b_2=-AC$, $n=-b_1b_2$  be as in lemma 4.1. By taking $M=1$, and $e=b,$ we have
\begin{equation*}
\begin{array}{l}
b_1M^4=BD\\
\\
b_2e^4=-b^4AC.\\
\end{array}
\end{equation*}
Then adding them up we get
\begin{align}
K=b_1M^4+b_2e^4&=(b^8+2b^6a^2+11b^4a^4+2b^2a^6+a^8)(b^8-b^4a^4+a^8)\\&\ \ -b^4(b^4+6b^2a^2+a^4)(b^8-4b^6a^2+8b^4a^4-4b^2a^6+a^8)\nonumber
\end{align}
Now, using Sage to factor K we get $K=a^4(a^6+b^2a^4+4b^4a^3-5b^6)^2.$ Consequently, $N=a^2(a^6+b^2a^4+4b^4a^3-5b^6)$. Since $\phi(G) $ is a subgroup of ${{\Bbb{Q}}^\ast}/{{\Bbb{Q}}^{\ast 2}},$ we get
\begin{equation}
\phi(G)\supseteq\{1, -n, -1, n, b_1, -b_1, b_2, -b_2 \}.
\end{equation}
On the other hand, for the curve
\begin{equation*}
y^2=x^3+4nx
\end{equation*}
we have
\begin{equation}
\psi(\overline{G})\supseteq\{1,n,2,2n\}.
\end{equation}
Again the numbers $1$ and $n$ are immediate consequence of the definition of the map $\psi$. For the numbers $2$ and $2n$ we note that the homogeneous equation
\begin{equation*}
N^2=2M^4+2ne^4
\end{equation*}
has the solution $M=p+q$, $e=1$, and $N=2(p^2+pq+q^2)$, where $n=p^4+q^4$. From Corollary $(4.2)$, we know that the right hand side of $(4.2)$, $(4.3)$ are independent modulo ${\Bbb Q}^{\ast 2}$.
 Therefore from these observations together with Eq. $(3.2)$ we get
\begin{equation*}
2^{r+2}=|\phi(G)||\psi(\overline{G}|\geq 4\cdot8=32.
\end{equation*}
This implies that $r\geq3$. But from $\omega(E_n)=1,$ the rank should be even. Therefore we see that $r$ is even and $r\geq 4$.
\subsection{Remark} If n is an even  number $n$ written in two different ways as sums of two biquadrates, then since $\omega(E_n)=-1$
in this case, the rank is odd and $r\geq 3$.

\section{Numerical Examples}
We conclude this paper by providing many examples of ranks $4$, $5$, $6$, $7$, $8$ and $10$ using sage software \cite{sage}.

{\begin{center}
\begin{table}[h]
\caption{Curves with even rank}\label{eqtable}
\renewcommand\arraystretch{1.5}
\begin{tabular}{|c|c|c|c|}\hline
$p$&$q$&$n$&$rank$\\\hline
$114732$&$15209$&$173329443404113736737$&$10$\\\hline
$3494$&$1623$&$155974778565937$&$8$\\\hline
$43676$&$11447$&$3656080821185585057$&$8$\\\hline
$500508$&$338921$&$75948917104718865094177$&$8$\\\hline
$502$&$271$&$68899596497$&$6$\\\hline
$292$&$193$&$8657437697$&$6$\\\hline
$32187$&$6484$&$1075069703066384497$&$4$\\\hline
$7604$&$5181$&$4063780581008977$&$4$\\\hline
$133$&$134$&$635318657$&$4$\\\hline
\end{tabular}\\
\end{table}
\end{center}

\begin{center}
\begin{table}[h]
\caption{Curves with odd rank}\label{eqtable}
\renewcommand\arraystretch{1.5}
\begin{tabular}{|c|c|c|c|}\hline
$p$&$q$&$n$&$rank$\\\hline
$989727$&$161299$&$960213785093149760746642$&$7$\\\hline
$129377$&$20297$&$280344024498199948322$&$7$\\\hline
$103543$&$47139$&$119880781585424489842$&$7$\\\hline
$119183$&$49003$&$207536518650314617202$&$7$\\\hline
$3537$&$661$&$156700232476402$&$7$\\\hline
$266063$&$72489$&$5038767537882101285602$&$5$\\\hline
$139361$&$66981$&$397322481336075317362$&$5$\\\hline
$38281$&$25489$&$2569595578866824162$&$5$\\\hline
\end{tabular}\\
\end{table}
\end{center}

\vspace{3cm}

\bibliographystyle{amsplain}

\end{document}